\documentclass[12pt, reqno]{amsart}
\usepackage{amsmath, amsthm, amscd, amsfonts, amssymb, graphicx, color}
\usepackage[bookmarksnumbered, colorlinks, plainpages]{hyperref}

\input{mathrsfs.sty}
\newtheorem{theorem}{Theorem}[section]
\newtheorem{lemma}[theorem]{Lemma}

\newtheorem{corollary}[theorem]{Corollary}
\theoremstyle{definition}

\theoremstyle{remark}
\newtheorem{remark}[theorem]{Remark}
\numberwithin{equation}{section}

\begin{document}
\title[Some Berezin number inequalities for  operator matrices]{Some Berezin number inequalities for  operator matrices}
\author[M. Bakherad  ]{Mojtaba Bakherad}
\address{Department of Mathematics, Faculty of Mathematics, University
of Sistan and Baluchestan, Zahedan, I.R.Iran.}
\email{mojtaba.bakherad@gmail.com; bakherad@member.ams.org}

\begin{abstract}
The Berezin symbol $\widetilde{A}$ of an operator $A$ acting on the reproducing
kernel Hilbert space ${\mathscr H}={\mathscr H(}\Omega)$ over some (non-empty) set is defined by
$\widetilde{A}(\lambda)=\langle A\hat{k}_{\lambda},\hat{k}_{\lambda}\rangle\,\,\,(\lambda\in\Omega)$,
where $\hat{k}_{\lambda}=\frac{{k}_{\lambda}}{\|{k}_{\lambda}\|}$
is the normalized reproducing kernel of ${\mathscr H}$. The  Berezin
number of operator $A$ is defined by
 $\mathbf{ber}(A)=\underset{\lambda \in
\Omega
}{\sup }\big|\tilde{A}(\lambda)\big|=\underset{\lambda \in
\Omega
}{\sup }\big|\langle A\hat{k}_{\lambda},\hat{k}_{\lambda}\rangle\big|$.
Moreover   $\mathbf{ber}(A)\leqslant w(A)$ (numerical radius).
In this paper, we present some Berezin number inequalities. Among other inequalities, it is shown that if
 $\mathbf{T}=\left[\begin{array}{cc}
 A&B\\
 C&D
 \end{array}\right]\in {\mathbb B}({\mathscr H(\Omega_1)}\oplus{\mathscr H(\Omega_2)})$, then
 \begin{align*}
 \mathbf{ber}(\mathbf{T}) \leqslant\frac{1}{2}\left( \mathbf{ber}(A)+ \mathbf{ber}(D)\right)+\frac{1}{2}\sqrt{\left( \mathbf{ber}(A)- \mathbf{ber}(D)\right)^2+(\|B\|+\|C\|)^2}.
 \end{align*}

\end{abstract}

\maketitle



\section{Introduction}
Let ${\mathscr H}$ be a complex
Hilbert space and ${\mathbb{B}}(\mathscr H)$ denote the $C^{\ast }$-algebra
of all bounded linear operators on ${\mathscr H}$ with the identity $I$. In
the case when {dim}${\mathscr H}=n$, we identify ${\mathbb{B}}({\mathscr H})$
with the matrix algebra $\mathbb{M}_{n}$ of all $n\times n$ matrices having
entries in the complex field. An operator $A\in{\mathbb B}({\mathscr H})$ is called positive
if $\langle Ax,x\rangle\geqslant0$ for all $x\in{\mathscr H }$, and then we write $A\geqslant0$. The numerical range and numerical radius of $A\in {\mathbb{B}}({%
\mathscr H})$ are defined by
\begin{equation*}
W(A):= \big\{\langle Ax,x\rangle :x\in {\mathscr H},\Vert
x\Vert =1\big\}\quad \textrm{and} \quad w(A):=\sup \big\{|\lambda |:\lambda\in  W(A)\big\},
\end{equation*}%
respectively. It is well known that $w(\,\cdot \,)$ defines a norm on ${\mathbb{B}}({%
\mathscr H})$, which is equivalent to the usual operator norm $\Vert\cdot
\Vert $. In fact, for any $A\in {\mathbb{B}}({\mathscr H})$, $\frac{1}{2}%
\Vert A\Vert \leqslant w(A)\leqslant \Vert A\Vert $ (see \cite[p. 9]{gof}). \\
A functional Hilbert space is a Hilbert space ${\mathscr H}={\mathscr H}(\Omega)$ of complex-valued
functions on a (non-empty) set $\Omega$, which has the property that point evaluations are continuous
i.e., for each $\lambda\in\Omega$ the map $f\longrightarrow f(\lambda)$ is a continuous linear functional on ${\mathscr H}$. Then the Riesz
representation theorem ensures that for each $\lambda\in\Omega$ there is a unique element $k_\lambda$ of ${\mathscr H}$ such that
$f(\lambda)=\langle f,k_\lambda\rangle$ for all $f\in{\mathscr H}$. The collection $\{k_\lambda:\lambda\in\Omega\}$ is called the reproducing kernel of ${\mathscr H}$. If $\{e_n\}$ is an orthonormal basis for a
functional Hilbert space ${\mathscr H}$, then the reproducing kernel of ${\mathscr H}$ is given by \cite[Problem 37]{hal}
\begin{equation*}
k_\lambda(z)=\sum_n\overline{e_n(\lambda)}e_n(z).
\end{equation*}%
For $\lambda\in\Omega$, let $\hat{k}_{\lambda}=\frac{{k}_{\lambda}}{\|{k}_{\lambda}\|}$ be the normalized reproducing kernel of ${\mathscr H}$. For a bounded linear operator $A$ on ${\mathscr H}$, the function $\widetilde{A}$ defined on $\Omega$ by
$\widetilde{A}(\lambda)=\langle A\hat{k}_{\lambda},\hat{k}_{\lambda}\rangle$
is the Berezin symbol of $A$, which firstly have been introduced by Berezin \cite{ber1, ber2}.  Berezin set and Berezin number of operator $A$ are defined by (see \cite{kar})
\begin{equation*}
 \mathbf{Ber}(A):=\big\{\widetilde{A}(\lambda):\lambda\in\Omega\big\}\quad\textrm{and} \quad  \mathbf{ber}(A):=\sup\big\{|\widetilde{A}(\lambda)|:\lambda\in \Omega\big\},
\end{equation*}%
respectively. It is clear that
the Berezin symbol $\widetilde{A}$  is the bounded function on $\Omega$ whose values lies in the numerical range of
the operator $A$ and hence%
\begin{align*}
 \mathbf{Ber}(A)\subseteq W(A)\qquad\textrm{and} \qquad  \mathbf{ber}(A)\leqslant w(A)
\end{align*}%
for all $A\in{\mathbb{B}}(\mathscr H)$. Karaev \cite{kar4} showed that if ${\mathscr H}^2$ is the Hardy space, then we take  $A = \langle\,\cdot\,,z\rangle z$ in ${\mathscr H}^2$, an elementary calculation
shows that $\widetilde{A}(\lambda)={|\lambda|^2}{(1-|\lambda|^2)}$, and thus
\begin{align*}
 \mathbf{Ber}(A) = \left[0, \frac{1}{4}\right]\varsubsetneq[0, 1] = W(A)\qquad\textrm{and} \qquad  \mathbf{ber}(A) = \frac{1}{4} \lneq 1 = w(A).
\end{align*}%

Moreover, Berezin number of an operator $A$  satisfies the following properties:\\
$(a)\,\,\mathbf{ber}(\alpha A)=|\alpha| \mathbf{ber}(A)$ for all $\alpha\in \mathcal{C}.$
\\
$(b)\,\, \mathbf{ber}(A+B)\leqslant  \mathbf{ber}(A)+ \mathbf{ber}(B).$
\\

Let $T_i\in{\mathbb{B}}(\mathscr H(\Omega))\,\,(1\leqslant i\leqslant n)$. Then we define the generalized Euclidean Berezin number of $T_{1},\cdots,T_{n}$ as follows
\begin{align*}
\mathbf{ber}_{p}(T_{1},\ldots,T_{n}):=\underset{\lambda\in\Omega
}{\sup }  \left(\sum_{i=1}^{n} \left| \left\langle T_{i}\hat{k}_{\lambda}, \hat{k}_{\lambda}\right\rangle \right|^{p}\right)^{\frac{1}{p}}.
\end{align*}
The generalized Euclidean Berezin number $\mathbf{ber}_{p}\,\,(p\geqslant1)$  has the following properties:\\
$(a)\,\,\mathbf{ber}_{p}(\alpha T_{1},\ldots,\alpha T_{n})=|\alpha|\,\mathbf{ber}_{p}(T_{1},\ldots,T_{n})$ for all $\alpha\in\mathcal{C}$;
\\
$(b)\,\,\mathbf{ber}_{p}(T_{1}+S_1,\ldots,T_{n}+S_n)\leqslant\mathbf{ber}_{p}(T_{1},\ldots,T_{n})+\mathbf{ber}_{p}(S_{1},\ldots,S_{n})$,
\\
where $T_i, S_i\in{\mathbb{B}}(\mathscr H(\Omega))\,\,(1\leqslant i\leqslant n)$.\\
 Namely, the Berezin symbol have been investigated in detail for the Toeplitz and Hankel operators
on the Hardy and Bergman spaces; it is widely applied in the various questions of analysis and uniquely determines
the operator (i.e., $\widetilde{A}(\lambda)=\widetilde{B}(\lambda)$ for all $\lambda\in\Omega$ implies $A = B$). For further information about Berezin symbol we refer the reader to \cite{kar2, kar3, kar4, nor, zhu} and references therein.\\
Let ${\mathscr H_{1}},{\mathscr H_{2}}, \cdots,{\mathscr H_{n}} $ be Hilbert spaces, and consider the direct sum ${\mathscr H}=\bigoplus_{j=1}^{n}{\mathscr H_{j}}$. With respect to this decomposition, every operator $T\in {\mathbb B}({\mathscr H})$ has an $n\times n$ operator matrix representation $T=[T_{ij}]$ with entries  $T_{ij}\in {\mathbb B}({\mathscr H_{j}}, {\mathscr H_{i}})$, the space of all bounded linear operators from ${\mathscr H_{j}}$ to ${\mathscr H_{i}}$.
Let $A\in {\mathbb B}({\mathscr H_{1}}, {\mathscr H_{1}})$, $B\in {\mathbb B}({\mathscr H_{2}}, {\mathscr H_{1}})$, $C\in {\mathbb B}({\mathscr H_{1}}, {\mathscr H_{2}})$ and $D\in {\mathbb B}({\mathscr H_{2}}, {\mathscr H_{2}})$. The operator matrix $\left[\begin{array}{cc} A&0\\ 0&D \end{array}\right]$ is called the diagonal part of $\left[\begin{array}{cc} A&B\\ C&D \end{array}\right]$ and $\left[\begin{array}{cc}  0&B\\ C&0 \end{array}\right]$ is the off-diagonal part.
Operator matrices provide a usual tool for studying Hilbert space operators, which have been extensively studied in the literatures.
J.C. Hou et al. \cite{hou} and  A. Omer et al. \cite{aA} established useful estimates for the spectral radius, the numerical radius, and the operator norm of an $n\times n$ operator matrix $\mathbf{T}=\left[T_{ij}\right]$. In particular, they proved that%
$$r(\mathbf{T})\leqslant r\left(\left[\|T_{ij}\|\right]\right),\quad w(\mathbf{T})\leqslant w\left(\left[\|T_{ij}\|\right]\right), \quad\|\mathbf{T}\|\leqslant \left\|\left[\|T_{ij}\|\right]\right\|$$%
{and}
$w(\mathbf{T})\leqslant w\left([t_{ij}]\right)$,
where $t_{ij}=\left\{\begin{array}{cc}
  w(T_{ij})&\textrm{if}\,\,\,\, i=j,\\
 \|T_{ij}\|&\textrm{if}\,\,\, i\not=j.
 \end{array}\right.$

The Berezin number is named in honor of F. Berezin, who introduced this concept in \cite{ber1}.
In this paper, we establish some inequalities involving the Berezin number of operators. By using the some ideas of \cite{aA, sheikh} we show several upper bounds for the Berezin number and the generalized Euclidean Berezin number  of  Hilbert space operators.

\section{main results}
Now we are in a position to present our first result.
\begin{theorem}\label{main1}
\label{kit} Let $\mathbf{T}=\left[T_{ij}\right]$ be $n\times n$ operator matrix with $T_{ij}\in {\mathbb B}({\mathscr H(\Omega_j)},{\mathscr H(\Omega_i)})\,\,(1\leqslant i,j\leqslant n)$. Then

\begin{equation*}
 \mathbf{ber}(\mathbf{T})\leqslant w\left([t_{ij}]\right),
\end{equation*}%
where $t_{ij}=\left\{\begin{array}{cc}
  \mathbf{ber}(T_{ij})&\textrm{if}\,\,\,\, i=j,\\
 \|T_{ij}\|&\textrm{if}\,\,\, i\not=j.
 \end{array}\right.$
\end{theorem}

\begin{proof}
Let ${\mathscr H}=\oplus_{i=1}^n{\mathscr H(\Omega_i)}$.
For every $(\lambda_1,\ldots,\lambda_n)\in\Omega_1\times\cdots\times\Omega_n$, let $\mathbf{\hat{k}_{(\lambda_1,\ldots,\lambda_n)}}=\left[\begin{array}{cc}
 k_{\lambda_1}\\
 \vdots\\
 k_{\lambda_n}
 \end{array}\right]$ be the normalized reproducing kernel of ${\mathscr H}$. Then
 \begin{align*}
 \left|\widetilde{\mathbf{T}}(\lambda_1,\cdots,\lambda_n)\right|&=
 \left|\left\langle\mathbf{T}\mathbf{\hat{k}_{(\lambda_1,\cdots,\lambda_n)}},\mathbf{\hat{k}_{(\lambda_1,\cdots,\lambda_n)}}\right\rangle\right|
 \\&=\left|\sum_{i,j=1}^n\left\langle T_{ij}k_{\lambda_j},k_{\lambda_i}\right\rangle\right|
 \\&\leqslant\sum_{i,j=1}^n\left|\left\langle T_{ij}k_{\lambda_j},k_{\lambda_i}\right\rangle\right|
 \\&=\sum_{i=1}^n\left|\left\langle T_{ii}k_{\lambda_i},k_{\lambda_i}\right\rangle\right|+\sum_{i,j=1\atop{i\neq j}}^n\left|\left\langle T_{ij}k_{\lambda_j},k_{\lambda_i}\right\rangle\right|
 \\&\leqslant\sum_{i=1}^n \mathbf{ber}(T_{ii})\|k_{\lambda_i}\|^2+\sum_{i,j=1\atop{i\neq j}}^n\|T_{ij}\|\|k_{\lambda_j}\|\|k_{\lambda_i}\|
 \\&=\sum_{i,j=1}^nt_{ij}\|k_{\lambda_j}\|\|k_{\lambda_i}\|
 \\&=\left\langle[t_{ij}]y,y\right\rangle,
 \end{align*}
 where $y=\left[\begin{array}{cc}
 \|k_{\lambda_1}\|\\
 \vdots\\
 \|k_{\lambda_n}\|
 \end{array}\right]$. It follows from $\|y\|=1$, that $\left|\widetilde{\mathbf{T}}(\lambda_1,\ldots,\lambda_n)\right|\leqslant w([t_{ij}])$.
 Hence
  \begin{align*}
 \mathbf{ber}(\mathbf{T})=\underset{(\lambda_1,\ldots,\lambda_n) \in
\Omega_1\times\cdots\times\Omega_n
}{\sup}
 \left|\widetilde{\mathbf{T}}(\lambda_1,\ldots,\lambda_n)\right|\leqslant w([t_{ij}])
  \end{align*}
 as required.
\end{proof}
\begin{corollary}\label{corcor}
 If $\mathbf{T}=\left[\begin{array}{cc}
 A&B\\
 C&D
 \end{array}\right]\in {\mathbb B}({\mathscr H(\Omega_1)}\oplus{\mathscr H(\Omega_2)})$, then
 \begin{align*}
 \mathbf{ber}(\mathbf{T}) \leqslant\frac{1}{2}\left( \mathbf{ber}(A)+ \mathbf{ber}(D)\right)+\frac{1}{2}\sqrt{\left( \mathbf{ber}(A)- \mathbf{ber}(D)\right)^2+(\|B\|+\|C\|)^2}.
 \end{align*}
In particular, for $\mathbf{T}=\left[\begin{array}{cc}
 A&0\\
 0&D
 \end{array}\right]\in {\mathbb B}({\mathscr H(\Omega_1)}\oplus{\mathscr H(\Omega_2)})$ we have
\begin{align}\label{max}
  \mathbf{ber}(\mathbf{T}) \leqslant\max\left\{ \mathbf{ber}(A), \mathbf{ber}(D)\right\}.
 \end{align}
\end{corollary}
\begin{proof}
Using Theorem \ref{main1} we get the inequality
\begin{align*}
 \mathbf{ber}\left(\left[\begin{array}{cc}
 A&B\\
 C&D
 \end{array}\right]\right)&\leqslant w\left(\left[\begin{array}{cc}
 \mathbf{ber}A&\|B\|\\
 \|C\|&\mathbf{ber}D
 \end{array}\right]\right)\\&= \frac{1}{2}r\left(\left[\begin{array}{cc}
 \mathbf{ber}A&\|B\|+\|C\|\\
 \|B\|+\|C\|&\mathbf{ber}D
 \end{array}\right]\right)\qquad(\textrm{by\, \cite[p. 44]{horn}})\\&= \frac{1}{2}\left( \mathbf{ber}(A)+ \mathbf{ber}(D)\right)+\frac{1}{2}\sqrt{\left( \mathbf{ber}(A)- \mathbf{ber}(D)\right)^2+(\|B\|+\|C\|)^2}.
 \end{align*}
 In particular, if $\mathbf{T}=\left[\begin{array}{cc}
 A&0\\
 0&D
 \end{array}\right]\in {\mathbb B}({\mathscr H(\Omega_1)}\oplus{\mathscr H(\Omega_2)})$, then
 \begin{align*}
 \mathbf{ber}(\mathbf{T}) &\leqslant\frac{ \mathbf{ber}(A)+ \mathbf{ber}(D)+\sqrt{\left( \mathbf{ber}(A)- \mathbf{ber}(D)\right)^2}}{2}\\&=\frac{ \mathbf{ber}(A)+ \mathbf{ber}(D)+\left| \mathbf{ber}(A)- \mathbf{ber}(D)\right|}{2}\\&=\max\left\{ \mathbf{ber}(A), \mathbf{ber}(D)\right\}.
 \end{align*}
\end{proof}
\bigskip We need the following lemmas for the next results. The next lemma follows from the spectral theorem for positive operators and Jensen inequality; see \cite{KIT}.
\begin{lemma}(The McCarty inequality)\label{3}
 Let $T\in{\mathbb B}({\mathscr H})$, $ T \geqslant 0$ and $x\in {\mathscr H}$ such that $\|x\|\leqslant1$. Then
 \begin{align*}
\left\langle Tx, x\right\rangle^{r} \leqslant \left\langle T^{r}x, x\right\rangle
\end{align*}
 for $r\geqslant1$.
\end{lemma}
\begin{proof}
Let $ r\geqslant 1$ and  $x\in {\mathscr H}$ such that $\|x\|\leqslant1$. Fix $u=\frac{x}{\|x\|}$. Using the McCarty inequality we have
$\left\langle Tu, u\right\rangle^{r} \leqslant  \left\langle T^{r}u, u\right\rangle$, whence
\begin{align*}
\left\langle Tx, x\right\rangle^{r} &\leqslant \|x\|^{2r-2} \left\langle T^{r}x, x\right\rangle\\&\leqslant\left\langle T^{r}x, x\right\rangle\qquad(\textrm {since\,}\|x\|\leqslant1\,\textrm {and\,}2r-2\geqslant0).
\end{align*}
Hence, we get the desired result.
\end{proof}
\begin{lemma}\cite[Theorem 1]{KIT}\label{5}
Let $T\in{\mathbb B}({\mathscr H})$ and $x, y\in {\mathscr H}$ be any vectors.  If $f$, $g$ are nonnegative  continuous functions on $[0, \infty)$ which are satisfying the relation $f(t)g(t)=t\,(t\in[0, \infty))$, then
\begin{align*}
| \left\langle Tx, y \right\rangle |^2 \leqslant \left\langle f^2(|T |)x ,x \right\rangle\, \left\langle g^2(| T^{*}|)y,y\right\rangle,
 \end{align*}
 in which $|T|=\left(T^*T\right)^{\frac{1}{2}}$.
 \end{lemma}
\begin{theorem}\label{main1}
Let
$\mathbf{T}=\left[\begin{array}{cc}
 0&X\\
 Y&0
 \end{array}\right]\in {\mathbb B}({\mathscr H(\Omega_1)\oplus\mathscr H(\Omega_2)})$,  $r\geqslant 1$ and $f$, $g$ be nonnegative  continuous  functions on $[0, \infty)$ satisfying the relation $f(t)g(t)=t\,(t\in[0, \infty))$. Then
 \begin{align*}
  \mathbf{ber}^{r}(\mathbf{T})\leqslant 2^{r-2} \mathbf{ber}^\frac{1}{2}\left(f^{2r}(|X|)+g^{2r}(|Y^*|)\right) \mathbf{ber}^\frac{1}{2}\left(f^{2r}(|Y|)+g^{2r}(|X^*|)\right)
   \end{align*}
   \end{theorem}
\begin{proof}
For every $(\lambda_1,\lambda_2)\in\Omega_1\times\Omega_2$, let $\mathbf{\hat{k}_{(\lambda_1,\lambda_2)}}=\left[\begin{array}{cc}
 k_{\lambda_1}\\
 k_{\lambda_2}
 \end{array}\right]$ be the normalized reproducing kernel of ${\mathscr H(\Omega_1)\oplus\mathscr H(\Omega_2)}$ (i.e., $\|k_{\lambda_1}\|^2+\|k_{\lambda_2}\|^2=1$). Then
 \begin{align*}
 &\left|\widetilde{\mathbf{T}}(\lambda_1,\lambda_2)\right|^r\\&
 =\left|\left\langle \mathbf{T}\mathbf{\hat{k}_{(\lambda_1,\lambda_2)}}, \mathbf{\hat{k}_{(\lambda_1,\lambda_2)}} \right\rangle \right|^{r}\\&
 =|\left\langle Xk_{\lambda_2}, k_{\lambda_1} \right\rangle+\left\langle Yk_{\lambda_1}, k_{\lambda_2} \right\rangle |^{r}\\&
 \leqslant\left(|\left\langle Xk_{\lambda_2}, k_{\lambda_1} \right\rangle|+|\left\langle Yk_{\lambda_1}, k_{\lambda_2} \right\rangle |\right)^{r} \qquad (\textrm {by the triangular inequality})\\&
 \leqslant\frac{2^r}{2}\left(|\left\langle Xk_{\lambda_2}, k_{\lambda_1} \right\rangle|^r+|\left\langle Yk_{\lambda_1}, k_{\lambda_2} \right\rangle |^{r}\right)
 \qquad (\textrm {by the convexity\,} f(t)=t^r)\\&
 \leqslant\frac{2^r}{2}\Big(\left(\left\langle f^2(|X|)k_{\lambda_2}, k_{\lambda_2} \right\rangle^\frac{1}{2}\left\langle g^2(|X^*|)k_{\lambda_1}, k_{\lambda_1} \right\rangle^\frac{1}{2}\right)^r
 \\&\qquad+\left(\left\langle f^2(|Y|)k_{\lambda_1}, k_{\lambda_1} \right\rangle^\frac{1}{2}\left\langle g^2(|Y^*|)k_{\lambda_2}, k_{\lambda_2} \right\rangle^\frac{1}{2} \right)^{r}\Big)
\qquad(\textrm {by Lemma\,\,}\ref{5})\\&\leqslant\frac{2^r}{2}\left(\left\langle f^{2r}(|X|)k_{\lambda_2}, k_{\lambda_2} \right\rangle^\frac{1}{2}\left\langle g^{2r}|X^*|k_{\lambda_1}, k_{\lambda_1} \right\rangle^\frac{1}{2}
 +\left\langle f^{2r}(|Y|)k_{\lambda_1}, k_{\lambda_1} \right\rangle^\frac{1}{2}\left\langle g^{2r}(|Y^*|)k_{\lambda_2}, k_{\lambda_2} \right\rangle^\frac{1}{2}\right)\\&
 \qquad\qquad\qquad\qquad\qquad\qquad\qquad\qquad\qquad (\textrm {by Lemma\,\,\ref{3})})\\&
 \leqslant\frac{2^r}{2}\left(\left\langle f^{2r}(|X|)k_{\lambda_2}, k_{\lambda_2} \right\rangle+\left\langle g^{2r}(|Y^*|)k_{\lambda_2}, k_{\lambda_2} \right\rangle\right)^\frac{1}{2}\\&\,\,\,\times
 \left(\left\langle f^{2r}(|Y|)k_{\lambda_1}, k_{\lambda_1} \right\rangle+\left\langle g^{2r}(|X^*|)k_{\lambda_1}, k_{\lambda_1} \right\rangle\right)^\frac{1}{2}
  \,\,\, (\textrm {by the Cauchy-Schwarz inequality})\\&
  =\frac{2^r}{2}\left\langle (f^{2r}(|X|)+g^{2r}(|Y^*|))k_{\lambda_2}, k_{\lambda_2} \right\rangle^\frac{1}{2} \left\langle (f^{2r}(|Y|)+g^{2r}(|X^*|))k_{\lambda_1}, k_{\lambda_1} \right\rangle^\frac{1}{2} \\&
 \leqslant\frac{2^r}{2} \mathbf{ber}^\frac{1}{2}\left(f^{2r}(|X|)+g^{2r}(|Y^*|)\right) \mathbf{ber}^\frac{1}{2}\left(f^{2r}(|Y|)+g^{2r}(|X^*|)\right)\|k_{\lambda_1}\|\|k_{\lambda_2}\|\\&
 \leqslant\frac{2^r}{2} \mathbf{ber}^\frac{1}{2}\left(f^{2r}(|X|)+g^{2r}(|Y^*|)\right) \mathbf{ber}^\frac{1}{2}
 \left(f^{2r}(|Y|)+g^{2r}(|X^*|)\right)\left(\frac{\|k_{\lambda_1}\|^2+\|k_{\lambda_2}\|^2}{2}\right)
 \\&\qquad\qquad\qquad\qquad\qquad\qquad\qquad(\textrm {by the arithmetic-geometric mean inequality})\\&
 =\frac{2^r}{4} \mathbf{ber}^\frac{1}{2}\left(f^{2r}(|X|)+g^{2r}(|Y^*|)\right) \mathbf{ber}^\frac{1}{2}\left(f^{2r}(|Y|)+g^{2r}(|X^*|)\right).
 \end{align*}
 Now, taking the supremum over all $(\lambda_1,\lambda_2)\in \Omega_1\times \Omega_2$ we get the desired result.
\end{proof}
\bigskip Theorem \ref{main1} includes a special case as follows.
\begin{corollary}\label{corol1}
Let
$\mathbf{T}=\left[\begin{array}{cc}
 0&X\\
 Y&0
 \end{array}\right]\in {\mathbb B}({\mathscr H_1\oplus\mathscr H_2})$, $0\leqslant p\leqslant1$ and $r\geqslant1$. Then
 \begin{align*}
  \mathbf{ber}^{r}(\mathbf{T})\leqslant 2^{r-2}   \mathbf{ber}^{\frac{1}{2}}\left(  | X |^{2rp} +   | Y^* |^{2r(1-p)} \right) \mathbf{ber}^{\frac{1}{2}}\left(| Y |^{2rp} +  | X^* |^{2r(1-p)} \right)
   \end{align*}
   \end{corollary}
  \begin{proof}
  The result follows immediately from Theorem \ref{main1} for $f(t)=t^p$ and $g(t)=t^{1-p}\,\,(0\leqslant p\leqslant1)$.
  \end{proof}
\begin{theorem}\label{main3}
Let
$A, B, X\in {\mathbb B}({\mathscr H}(\Omega))$,   $r\geqslant 1$ and $f$, $g$ be nonnegative  continuous  functions on $[0, \infty)$ satisfying the relation $f(t)g(t)=t\,(t\in[0, \infty))$. Then
 \begin{align*}
 \mathbf{ber}^{r}(A^*XB)\leqslant  \mathbf{ber}\left(\frac{1}{p}\left[B^*f^2(|X|)B\right]^{\frac{rp}{2}}+\frac{1}{q}\left[A^*g^2(|X^*|)A\right]^{\frac{rq}{2}}\right),
   \end{align*}
   where $\frac{1}{p}+\frac{1}{q}=1$ and $pr\geqslant qr\geqslant2$.
   \end{theorem}
\begin{proof}
For every $\lambda\in\Omega$, let $\mathbf{\hat{k}_{\lambda}}$ be the normalized reproducing kernel of ${\mathscr H}(\Omega)$. Then
\begin{align*}
 \left|\langle A^*XB\mathbf{\hat{k}_{\lambda}},\mathbf{\hat{k}_{\lambda}}\rangle\right|^r
 &= \left|\langle XB\mathbf{\hat{k}_{\lambda}},A\mathbf{\hat{k}_{\lambda}}\rangle\right|^r
 \\&\leqslant\left(\langle f^2(|X|)B\mathbf{\hat{k}_{\lambda}},B\mathbf{\hat{k}_{\lambda}}\rangle\langle g^2(|X^*|)A\mathbf{\hat{k}_{\lambda}},A\mathbf{\hat{k}_{\lambda}}\rangle\right)^{\frac{r}{2}}
 \\&
 \qquad\qquad\qquad\qquad\qquad\qquad\qquad\qquad\qquad (\textrm {by Lemma\,\,\ref{5}})
 \\&\leqslant\frac{1}{p}\langle f^2(|X|)B\mathbf{\hat{k}_{\lambda}},B\mathbf{\hat{k}_{\lambda}}\rangle^{\frac{rp}{2}}+\frac{1}{q}\langle g^2(|X^*|)A\mathbf{\hat{k}_{\lambda}},A\mathbf{\hat{k}_{\lambda}}\rangle^{\frac{rq}{2}}
 \\&
 \qquad\qquad\qquad\qquad\qquad\qquad\qquad\qquad\qquad (\textrm {by the Young inequaliy})
 \\&=\frac{1}{p}\langle B^*f^2(|X|)B\mathbf{\hat{k}_{\lambda}},\mathbf{\hat{k}_{\lambda}}\rangle^{\frac{rp}{2}}+\frac{1}{q}\langle A^*g^2(|X^*|)A\mathbf{\hat{k}_{\lambda}},\mathbf{\hat{k}_{\lambda}}\rangle^{\frac{rq}{2}}
 \\&\leqslant\frac{1}{p}\langle \left[B^*f^2(|X|)B\right]^{\frac{rp}{2}}\mathbf{\hat{k}_{\lambda}},\mathbf{\hat{k}_{\lambda}}\rangle+\frac{1}{q}\langle
 \left[A^*g^2(|X^*|)A\right]^{\frac{rq}{2}}\mathbf{\hat{k}_{\lambda}},\mathbf{\hat{k}_{\lambda}}\rangle\\&
 \qquad\qquad\qquad\qquad\qquad\qquad\qquad\qquad\qquad (\textrm {by Lemma\,\,\ref{3}})
 \\&=\left\langle \left(\frac{1}{p} \left[B^*f^2(|X|)B\right]^{\frac{rp}{2}}+\frac{1}{q}\left[A^*g^2(|X^*|)A\right]^{\frac{rq}{2}}\right)\mathbf{\hat{k}_{\lambda}},\mathbf{\hat{k}_{\lambda}}\right\rangle
  \\& \leqslant\mathbf{ber}\left(\frac{1}{p}\left[B^*f^2(|X|)B\right]^{\frac{rp}{2}}+\frac{1}{q}\left[A^*g^2(|X^*|)A\right]^{\frac{rq}{2}}\right),
   \end{align*}
   whence
   \begin{align*}
 \mathbf{ber}^r(A^*XB)=\underset{\lambda\in
\Omega
}{\sup}\left|\langle A^*XB\mathbf{\hat{k}_{\lambda}},\mathbf{\hat{k}_{\lambda}}\rangle\right|^r
  \leqslant\mathbf{ber}\left(\frac{1}{p}\left[B^*f^2(|X|)B\right]^{\frac{rp}{2}}+\frac{1}{q}\left[A^*g^2(|X^*|)A\right]^{\frac{rq}{2}}\right).
   \end{align*}
 \end{proof}
\begin{remark}
In Theorem \ref{main3}, if we add the hypophysis of contraction for operators $A$ and $B$ (i.e., $A^*A\leqslant I$ and $B^*B\leqslant I$) then by using Lemma \ref{3} and with a similar fashion in the proof in the Theorem \ref{main3} we get the inequality
\begin{align}\label{sheb}
 \mathbf{ber}^{r}(A^*XB)\leqslant  \mathbf{ber}\left(\frac{1}{p}B^*f^{rp}(|X|)B+\frac{1}{q}A^*g^{rq}(|X^*|)A\right),
   \end{align}
 where   $r\geqslant 1$, $\frac{1}{p}+\frac{1}{q}=1$ such that $pr\geqslant qr\geqslant2$ and $f$, $g$ be nonnegative  continuous  functions on $[0, \infty)$ satisfying the relation $f(t)g(t)=t\,(t\in[0, \infty))$.
\end{remark}
The next result follows  from Theorem \ref{main3} and inequality \eqref{sheb} for $f(t)=t^\alpha$ and $g(t)=t^{1-\alpha}\,\,(0\leqslant \alpha\leqslant1)$.
\begin{corollary}
Let
$A, B, X\in {\mathbb B}({\mathscr H}(\Omega))$,   $r\geqslant 1$, $\frac{1}{p}+\frac{1}{q}=1$ such that $pr\geqslant qr\geqslant2$ and $0\leqslant\alpha\leqslant1$. Then
 \begin{align*}
 \mathbf{ber}^{r}(A^*XB)\leqslant  \mathbf{ber}\left(\frac{1}{p}\left[B^*|X|^{2\alpha}B\right]^{\frac{rp}{2}}+\frac{1}{q}\left[A^*|X^*|^{2(1-\alpha)}A\right]^{\frac{rq}{2}}\right),
   \end{align*}
In particular, if $A$ and $B$ be contraction, then
\begin{align*}
 \mathbf{ber}^{r}(A^*XB)\leqslant  \mathbf{ber}\left(\frac{1}{p}B^*|X|^{rp\alpha}B+\frac{1}{q}A^*|X^*|^{rp(1-\alpha)}A\right).
   \end{align*}
   \end{corollary}
 Now, we need the following lemma for the next result.
\begin{lemma}\label{man1}
Let
$X, Y\in {\mathbb B}({\mathscr H}(\Omega))$. If $\mathbf{ber}\left(\left[\begin{array}{cc} X&0\\0&0 \end{array}\right]\right)\leqslant\mathbf{ber}\left(\left[\begin{array}{cc} Y&0\\0&0 \end{array}\right]\right)$, then
$\mathbf{ber}(X)\leqslant \mathbf{ber}(Y)$.
\end{lemma}
\begin{proof}
For every $\lambda\in\Omega$, let $\mathbf{\hat{k}_{\lambda}}$ be the normalized reproducing kernel of ${\mathscr H}(\Omega)$. Then
\begin{align*}
 \left|\left\langle X\mathbf{\hat{k}_{\lambda}},\mathbf{\hat{k}_{\lambda}}\right\rangle\right|&
 =\left|\left\langle\left[\begin{array}{cc} X&0\\0&0 \end{array}\right]\left[\begin{array}{cc} \mathbf{\hat{k}_{\lambda}}\\0 \end{array}\right],\left[\begin{array}{cc} \mathbf{\hat{k}_{\lambda}}\\0 \end{array}\right]\right\rangle\right|
 \\&\leqslant\mathbf{ber}\left(\left[\begin{array}{cc} X&0\\0&0 \end{array}\right]\right)\qquad (\textrm {by the definition of}\,\, \mathbf{ber} )
 \\&\leqslant\mathbf{ber}\left(\left[\begin{array}{cc} Y&0\\0&0 \end{array}\right]\right)\\&\leqslant  \mathbf{ber}\left(Y\right)
 \qquad (\textrm {by inequality} \eqref{max}).
   \end{align*}
   Hence
   \begin{align*}
 \mathbf{ber}\left(X\right)=\underset{\lambda\in
\Omega
}{\sup}\left|\langle X\mathbf{\hat{k}_{\lambda}},\mathbf{\hat{k}_{\lambda}}\rangle\right|
 \leqslant  \mathbf{ber}\left(Y\right).
   \end{align*}
\end{proof}
\begin{corollary}
Let
$A_i, B_i, X_i\in {\mathbb B}({\mathscr H}(\Omega))\,\,(1\leqslant i\leqslant n)$,   $r\geqslant 1$ and $f$, $g$ be nonnegative  continuous  functions on $[0, \infty)$ satisfying the relation $f(t)g(t)=t\,(t\in[0, \infty))$. Then
 \begin{align*}
 \mathbf{ber}^{r}\left(\sum_{i=1}^nA_i^*X_iB_i\right)\leqslant  \mathbf{ber}\left(\frac{1}{p}\left[\sum_{i=1}^nB_i^*f^2(|X_i|)B_i\right]^{\frac{rp}{2}}+\frac{1}{q}\left[\sum_{i=1}^nA_i^*g^2(|X_i^*|)A_i\right]^{\frac{rq}{2}}\right),
   \end{align*}
   where $\frac{1}{p}+\frac{1}{q}=1$ and $pr\geqslant qr\geqslant2$. \\
   In particular, if $\sum_{i=1}^nA^*_iA_i\leqslant I$ and $\sum_{i=1}^nB^*_iB_i\leqslant I$, then
   \begin{align*}
 \mathbf{ber}^{r}\left(\sum_{i=1}^nA_i^*X_iB_i\right)\leqslant  \mathbf{ber}\left(\frac{1}{p}\sum_{i=1}^nB_i^*f^{rp}(|X_i|)B_i+\frac{1}{q}\sum_{i=1}^nA_i^*g^{rq}(|X_i^*|)A_i\right).
   \end{align*}
   \end{corollary}
   \begin{proof}
 If we replace $A$, $B$ and $X$ by operator matrices
 \begin{align*}
\left[\begin{array}{cccc}
A_1&0&\cdots&0\\
 A_2&0&\cdots&0\\
\vdots&\vdots&\ddots&\vdots\\
A_n&0&\cdots&0\\
\end{array}\right], \left[\begin{array}{cccc}
B_1&0&\cdots&0\\
B_2&0&\cdots&0\\
\vdots&\vdots&\ddots&\vdots\\
B_n&0&\cdots&0\\
\end{array}\right]\,\, \textrm{and}\,\, \left[\begin{array}{cccc}
X_1&0&\cdots&0\\
0&X_2&\cdots&0\\
\vdots&\vdots&\ddots&\vdots\\
0&0&\cdots&X_n\\
\end{array}\right],
\end{align*}
respectively, in Theorem \ref{main3}, then get
\begin{align*}
 &\mathbf{ber}^{r}\left(\left[\begin{array}{cc}
 \sum_{i=1}^nA_i^*X_iB_i&0\\
 0&0
 \end{array}\right]\right)\\&\qquad\qquad\leqslant  \mathbf{ber}\left(\left[\begin{array}{cc}
 \frac{1}{p}\left[\sum_{i=1}^nB_i^*f^2(|X_i|)B_i\right]^{\frac{rp}{2}}+\frac{1}{q}\left[\sum_{i=1}^nA_i^*g^2(|X_i^*|)A_i\right]^{\frac{rq}{2}}&0\\
 0&0
 \end{array}\right]\right).
\end{align*}
Now, using Lemma \ref{man1} we have
\begin{align*}
 \mathbf{ber}^{r}\left(\sum_{i=1}^nA_i^*X_iB_i\right)\leqslant  \mathbf{ber}\left(\frac{1}{p}\left[\sum_{i=1}^nB_i^*f^2(|X_i|)B_i\right]^{\frac{rp}{2}}
 +\frac{1}{q}\left[\sum_{i=1}^nA_i^*g^2(|X_i^*|)A_i\right]^{\frac{rq}{2}}\right)
   \end{align*}
the first inequality. The second inequality follows from inequality \eqref{sheb} and this complete the proof.
   \end{proof}
   In the next theorem we present an inequality involving the generalized Euclidean Berezin number for off-diagonal operator matrices.
   \begin{theorem}
 Let
 $T_{i}=\left[\begin{array}{cc}
 0&X_{i}\\
 Y_{i}&0
 \end{array}\right]
 \in {\mathbb B}({\mathscr H(\Omega_1)\oplus\mathscr H(\Omega_2)})\,\,(1\leqslant i\leqslant n)$. Then
\begin{align*}
 \mathbf{ber}_{p}^{p}(T_{1}, T_{2},&\ldots,T_{n})\\&\leqslant 2^{p-2}\sum_{i=1}^{n}\mathbf{ber}^\frac{1}{2}\left(f^{2p}(|X_{i}|)+g^{2p}(|Y^*_{i}|)\right)
\mathbf{ber}^\frac{1}{2}\left(f^{2p}(|Y_{i}|)+g^{2p}(|X^*_{i}|)\right)
 \end{align*}
  for  $p\geqslant 1$.
 \end{theorem}
 \begin{proof}
 For every $(\lambda_1,\lambda_2)\in\Omega_1\times\Omega_2$, let $\mathbf{\hat{k}_{(\lambda_1,\lambda_2)}}=\left[\begin{array}{cc}
 k_{\lambda_1}\\
 k_{\lambda_2}
 \end{array}\right]$ be the normalized reproducing kernel of ${\mathscr H(\Omega_1)\oplus\mathscr H(\Omega_2)}$ (i.e., $\|k_{\lambda_1}\|^2+\|k_{\lambda_2}\|^2=1$). Then
 {\footnotesize\begin{align*}
 &\sum_{i=1}^{n}|\langle T_{i}\mathbf{\hat{k}_{(\lambda_1,\lambda_2)}}, \mathbf{\hat{k}_{(\lambda_1,\lambda_2)}}\rangle|^{p}
 \\&=\sum_{i=1}^{n}|\langle X_{i}k_{\lambda_2}, k_{\lambda_1}\rangle+\langle Y_{i}k_{\lambda_1}, k_{\lambda_2}\rangle|^{p}
  \\&\leqslant \sum_{i=1}^{n}\left(|\langle X_{i}k_{\lambda_2}, k_{\lambda_1}\rangle|+|\langle Y_{i}k_{\lambda_1}, k_{\lambda_2}\rangle|\right)^{p}
  \qquad (\textrm {by the triangular inequality})
 \\&\leqslant\frac{2^p}{2}\sum_{i=1}^{n}|\langle X_{i}k_{\lambda_2}, k_{\lambda_1}\rangle|^{p}+|\langle Y_{i}k_{\lambda_1}, k_{\lambda_2}\rangle|^{p}
 \qquad (\textrm {by the convexity\,} f(t)=t^p)
  \\&\leqslant\frac{2^p}{2}\sum_{i=1}^{n}\langle f^2(|X_{i}|)k_{\lambda_2}, k_{\lambda_2}\rangle^\frac{p}{2}\langle g^2(|X^*_{i}|)k_{\lambda_1}, k_{\lambda_1}\rangle^\frac{p}{2}
 +\langle f^2(|Y_{i}|)k_{\lambda_1}, k_{\lambda_1}\rangle^\frac{p}{2}\langle g^2(|Y^*_{i}|)k_{\lambda_2}, k_{\lambda_2}\rangle^\frac{p}{2}\\&
 \qquad\qquad\qquad\qquad\qquad\qquad\qquad\qquad\qquad\qquad(\textrm {by Lemma\,\,\ref{5}})
\\&\leqslant\frac{2^p}{2}\sum_{i=1}^{n}\langle f^{2p}(|X_{i}|)k_{\lambda_2}, k_{\lambda_2}\rangle^\frac{1}{2}\langle g^{2p}(|X^*_{i}|)k_{\lambda_1}, k_{\lambda_1}\rangle^\frac{1}{2}
+\langle f^{2p}(|Y_{i}|)k_{\lambda_1}, k_{\lambda_1}\rangle^\frac{1}{2}\langle g^{2p}(|Y^*_{i}|)k_{\lambda_2}, k_{\lambda_2}\rangle^\frac{1}{2}
\\&\qquad\qquad\qquad\qquad\qquad\qquad\qquad\qquad\qquad\qquad(\textrm {by Lemma\,\,\ref{3}})
\\&\leqslant\frac{2^p}{2}\sum_{i=1}^{n}\left(\left\langle f^{2p}(|X_{i}|)k_{\lambda_2}, k_{\lambda_2}\rangle+\langle g^{2p}(|Y^*_{i}|)k_{\lambda_2}, k_{\lambda_2}\right\rangle\right)^\frac{1}{2}
\left(\left\langle f^{2p}(|Y_{i}|)k_{\lambda_1}, k_{\lambda_1}\right\rangle+\left\langle g^{2p}(|X^*_{i}|)k_{\lambda_1}, k_{\lambda_1}\right\rangle\right)^\frac{1}{2}\\&
\qquad\qquad\qquad\qquad\qquad\qquad\qquad\qquad\qquad(\textrm {by the Cauchy-Schwarz inequality})
\\&\leqslant\frac{2^p}{2}\sum_{i=1}^{n}\mathbf{ber}^\frac{1}{2}\left(f^{2p}(|X_{i}|)+ g^{2p}(|Y^*_{i}|)\right)\mathbf{ber}^\frac{1}{2}\left(f^{2p}(|Y_{i}|)+g^{2p}(|X^*_{i}|)\right)\|k_{\lambda_1}\|\|k_{\lambda_2}\|
\\&=\frac{2^p}{2}\sum_{i=1}^{n}\mathbf{ber}^\frac{1}{2}\left(f^{2p}(|X_{i}|)+ g^{2p}(|Y^*_{i}|)\right)\mathbf{ber}^\frac{1}{2}\left(f^{2p}(|Y_{i}|)+g^{2p}(|X^*_{i}|)\right)
\left(\frac{\|k_{\lambda_1}\|^2+\|k_{\lambda_2}\|^2}{2}\right)\\&
=\frac{2^p}{4}\sum_{i=1}^{n}\mathbf{ber}^\frac{1}{2}\left(f^{2p}(|X_{i}|)+ g^{2p}(|Y^*_{i}|)\right)\mathbf{ber}^\frac{1}{2}\left(f^{2p}(|Y_{i}|)+g^{2p}(|X^*_{i}|)\right).
  \end{align*}}
  Hence
  \begin{align*}
 \mathbf{ber}_{p}^{p}(T_{1}, T_{2}&,\ldots,T_{n})=\underset{(\lambda_1, \lambda_2) \in
  \Omega_1\times \Omega_2
}{\sup}\sum_{i=1}^{n}|\langle T_{i}\mathbf{\hat{k}_{(\lambda_1,\lambda_2)}}, \mathbf{\hat{k}_{(\lambda_1,\lambda_2)}}\rangle|^{p}\\&\leqslant 2^{p-2}\sum_{i=1}^{n}\mathbf{ber}^\frac{1}{2}\left(f^{2p}(|X_{i}|)+g^{2p}(|Y^*_{i}|)\right)
\mathbf{ber}^\frac{1}{2}\left(f^{2p}(|Y_{i}|)+g^{2p}(|X^*_{i}|)\right)
 \end{align*}
 as required.
 \end{proof}
 \begin{theorem}\label{main4}
\label{th1}Let
 $T_{i}=\left[
\begin{array}{cc}
A_{i} & B_{i} \\
C_{i} & D_{i}%
\end{array}%
\right]\in {\mathbb B}({\mathscr H(}\Omega_1)\oplus{\mathscr H}(\Omega_2)) \,\,(1\leqslant i\leqslant n)$ and $p\geqslant 1$. Then
\begin{align*}
\mathbf{ber}_{p}^{p}&(T_{1},\ldots ,T_{n})\\&\leqslant 2^{-p}\sum_{i=1}^{n}\left(
\mathbf{ber} \left( A_{i}\right) +\mathbf{ber} \left( D_{i}\right) +\sqrt{\left( \mathbf{ber}
\left( A_{i}\right) -\mathbf{ber} \left( D_{i}\right) \right) ^{2}+\left(
\left\Vert B_{i}\right\Vert +\left\Vert C_{i}\right\Vert \right) ^{2}}%
\right) ^{p}.  \label{el2}
\end{align*}%
\end{theorem}
\begin{proof}
For every $(\lambda_1,\lambda_2)\in\Omega_1\times\Omega_2$, let $\mathbf{\hat{k}_{(\lambda_1,\lambda_2)}}=\left[\begin{array}{cc}
 k_{\lambda_1}\\
 k_{\lambda_2}
 \end{array}\right]$ be the normalized reproducing kernel of ${\mathscr H(\Omega_1)\oplus\mathscr H(\Omega_2)}$.
It follows from
\begin{align*}
\left\vert \left\langle T_{i}\mathbf{\hat{k}_{(\lambda_1,\lambda_2)}},\mathbf{\hat{k}_{(\lambda_1,\lambda_2)}}\right\rangle \right\vert & =\left\vert
\left\langle \left[
\begin{array}{cc}
A_{i} & B_{i} \\
C_{i} & D_{i}%
\end{array}%
\right] \left[
\begin{array}{c}
k_{\lambda_1} \\
k_{\lambda_2}%
\end{array}%
\right] ,\left[
\begin{array}{c}
k_{\lambda_1} \\
k_{\lambda_2}%
\end{array}%
\right] \right\rangle \right\vert \\
& =\left\vert \left\langle \left[
\begin{array}{c}
A_{i}k_{\lambda_1}+B_{i}k_{\lambda_2} \\
C_{i}k_{\lambda_1}+D_{i}k_{\lambda_2}%
\end{array}%
\right] ,\left[
\begin{array}{c}
k_{\lambda_1} \\
k_{\lambda_2}%
\end{array}%
\right] \right\rangle \right\vert \\
& =\left\vert \left\langle A_{i}k_{\lambda_1},k_{\lambda_1}\right\rangle +\left\langle
B_{i}k_{\lambda_2},k_{\lambda_1}\right\rangle +\left\langle C_{i}k_{\lambda_1},k_{\lambda_2}\right\rangle +\left\langle
D_{i}k_{\lambda_2},k_{\lambda_2}\right\rangle \right\vert \\
& \leqslant\left\vert \left\langle A_{i}k_{\lambda_1},k_{\lambda_1}\right\rangle \right\vert +\left\vert
\left\langle B_{i}k_{\lambda_2},k_{\lambda_1}\right\rangle \right\vert +\left\vert \left\langle
C_{i}k_{\lambda_1},k_{\lambda_2}\right\rangle \right\vert +\left\vert \left\langle
D_{i}k_{\lambda_2},k_{\lambda_2}\right\rangle \right\vert
\end{align*}%
that
{\footnotesize\begin{align*}
&\mathbf{ber}_{p}^{p}(T_{1},\ldots ,T_{n})\\& =\sup_{(\lambda_1,\lambda_2)\in\Omega_1\times\Omega_2}\sum_{i=1}^{n}\left\vert \left\langle T_{i}\mathbf{\hat{k}_{(\lambda_1,\lambda_2)}},\mathbf{\hat{k}_{(\lambda_1,\lambda_2)}}\right\rangle \right\vert
^{p} \\
& \leqslant \sup_{(\lambda_1,\lambda_2)\in\Omega_1\times\Omega_2}\sum_{i=1}^{n}\left(
\left\vert \left\langle A_{i}k_{\lambda_1},k_{\lambda_1}\right\rangle \right\vert +\left\vert
\left\langle B_{i}k_{\lambda_2},k_{\lambda_1}\right\rangle \right\vert +\left\vert \left\langle
C_{i}k_{\lambda_1},k_{\lambda_2}\right\rangle \right\vert +\left\vert \left\langle
D_{i}k_{\lambda_2},k_{\lambda_2}\right\rangle \right\vert \right) ^{p} \\
& \leqslant \sum_{i=1}^{n}\left( \sup_{(\lambda_1,\lambda_2)\in\Omega_1\times\Omega_2}
\left( \left\vert \left\langle A_{i}k_{\lambda_1},k_{\lambda_1}\right\rangle \right\vert
+\left\vert \left\langle B_{i}k_{\lambda_2},k_{\lambda_1}\right\rangle \right\vert +\left\vert
\left\langle C_{i}k_{\lambda_1},k_{\lambda_2}\right\rangle \right\vert +\left\vert \left\langle
D_{i}k_{\lambda_2},k_{\lambda_2}\right\rangle \right\vert \right) \right) ^{p} \\
& \leqslant\sum_{i=1}^{n}\left( \sup_{(\lambda_1,\lambda_2)\in\Omega_1\times\Omega_2}
\left( \mathbf{ber} \left( A_{i}\right) \left\Vert k_{\lambda_1}\right\Vert ^{2}+\mathbf{ber}
\left( D_{i}\right) \left\Vert k_{\lambda_2}\right\Vert ^{2}+\left( \left\Vert
B_{i}\right\Vert +\left\Vert C_{i}\right\Vert \right) \left\Vert
k_{\lambda_1}\right\Vert \left\Vert k_{\lambda_2}\right\Vert \right) \right) ^{p} \\
&  \leqslant \sum_{i=1}^{n}\left( \sup_{\theta \in \left[ 0,2\pi \right] }\left(
\mathbf{ber} \left( A_{i}\right) \cos^2 \theta +\mathbf{ber} \left( D_{i}\right) \sin^2
\theta +\left( \left\Vert B_{i}\right\Vert +\left\Vert C_{i}\right\Vert
\right) \cos \theta \sin \theta \right) \right) ^{p} \\
&=2^{-p}\sum_{i=1}^{n}\left( \mathbf{ber} \left( A_{i}\right) +\mathbf{ber} \left(
D_{i}\right) +\sqrt{\left( \mathbf{ber} \left( A_{i}\right) -\mathbf{ber} \left(
D_{i}\right) \right) ^{2}+\left( \left\Vert B_{i}\right\Vert +\left\Vert
C_{i}\right\Vert \right) ^{2}}\right) ^{p}.
\end{align*}}%
This completes the proof.
\end{proof}

\textbf{Acknowledgement.} The author would like to thank the Tusi Mathematical Research Group (TMRG).
\bigskip


\begin{thebibliography}{99}

\bibitem{aA} A. Abu-Omar and F. Kittaneh,  \textit{Numerical radius inequalities for $n\times n$
operator matrices}, Linear Algebra Appl. 468 (2015), 18--26.

\bibitem{ber1} F.A. Berezin, \textit{Covariant and contravariant symbols for operators},
Math. USSR-Izv. \textbf{6} (1972) 1117--1151.

\bibitem{ber2}  F.A. Berezin, \textit{Quantization}, Math. USSR-Izv.
\textbf{8} (1974) 1109--1163.

\bibitem{gof} K.E. Gustafson and D.K.M. Rao, \textit{Numerical Range, The
Field of Values of Linear Operators and Matrices}, Springer, New York, 1997.

\bibitem{hal} P.R. Halmos, \textit{A Hilbert Space Problem Book}, Springer-Verlag, New York, 1982.

\bibitem{hed} H. Hedenmalm, B. Korenblum and K. Zhu, \textit{Theory of Bergman Spaces},
Grad. Texts in Math., Springer-Verlag, Berlin,
2000.

\bibitem{horn} R.A. Horn and C.R. Johnson, \textit{Topics in Matrix Analysis},
Cambridge University Press, Cambridge, 1991.

\bibitem{hou} J.C. Hou and H.K. Du, \textit{Norm inequalities of positive operator matrices},
Integral Equations Operator Theory \textbf{22} (1995) 281--294.

\bibitem{kar} M.T. Karaev,  \textit{Berezin symbol and invertibility of operators on
the functional Hilbert spaces}, J. Funct. Anal., \textbf{238} (2006) 181--192.

\bibitem{kar2} M.T. Karaev, \textit{On the Berezin symbol}, J. Math. Sci. (New York)
 \textbf{115} (2003) 2135--2140. Translated from: Zap.
Nauchn. Sem. S.-Peterburg. Otdel. Mat. Inst. Steklov. (POMI) \textbf{270} (2000) 80--89.

\bibitem{kar3} M.T. Karaev, \textit{Functional analysis proofs of Abel's theorems},
Proc. Amer. Math. Soc. \textbf{132} (2004) 2327--2329.

\bibitem{kar4} M.T. Karaev and S. Saltan, \textit{Some results on Berezin symbols},
Complex Var. Theory Appl. \textbf{50} (3) (2005) 185-–193.

\bibitem{KIT} F. Kittaneh, \textit{Notes on some inequalitis for Hilbert space operators},
Publ. Res. Inst. Math. Sci. \textbf{24} (2) (1988), 283--293.

\bibitem{nor} E. Nordgren and P. Rosenthal, \textit{Boundary values of Berezin symbols},
 Oper. Theory Adv. Appl. \textbf{73} (1994) 362--368.

\bibitem{sheikh} A. Sheikhhosseini, M.S. Moslehian and K. Shebrawi, \textit{%
Inequalities for generalized Euclidean operator radius via Young's inequality%
}, J. Math. Anal. Appl. \textbf{445} (2017) no. 2, 1516--1529.

\bibitem{zhu} K. Zhu, \textit{Operator Theory in Function Spaces}, Dekker, New York, 1990.


\end{thebibliography}
\end{document}